\documentclass[12pt]{amsart}
\usepackage{amscd,amssymb,hyperref,amsfonts,xcolor,graphicx, epsfig}
\usepackage[all]{xy}
\usepackage{cite}
\let\OLDthebibliography\thebibliography
\renewcommand\thebibliography[1]{
  \OLDthebibliography{#1}
  \setlength{\parskip}{1pt}
  \setlength{\itemsep}{1pt plus 0.3ex}
}

\newtheorem{theorem}{Theorem}
\newtheorem{cor}{Corollary}

\newtheorem{prop}{Proposition}

\theoremstyle{definition}

\newtheorem{example}{Example}
\theoremstyle{remark}

\begin{document}
\title{Multi-switches and representations of braid groups}

\author{Valeriy Bardakov,~Timur Nasybullov}
\date{}
\begin{abstract}
In the paper we introduce the notion of a (virtual) multi-switch which generalizes the notion of a (virtual) switch. Using (virtual) multi-switches we introduce a general approach how to construct representations of (virtual) braid groups by automorphisms of algebraic systems. As a corollary we introduce new representations of virtual braid groups which generalize several previously known representations.

~
\end{abstract}
\maketitle
\section{Introduction}
\textit{A set-theoretical solution of the Yang-Baxter equation} is a pair $(X, S)$, where $X$ is a set and $S:X\times X\to X\times X$ is a bijective map such that
$$(S\times id)(id \times S)(S\times id)=(id\times S)(S \times id)(id\times S).$$
The problem of studying set-theoretical solutions of the Yang-Baxter equation was formulated by Drinfel'd in \cite{Dri}. If $(X, S)$ is a set-theoretical solution of the Yang-Baxter equation, then the map $S$ is called \textit{a switch} on $X$ (see \cite{FMK}). A pair of switches $(S,V)$ on $X$ is called \textit{a virtual switch} on $X$ if $V^2=id$ and the equality
$$(V\times id)(id \times S)(V\times id)=(id\times V)(S \times id)(id\times V)$$
holds.

Switches and virtual switches are strongly connected with virtual braid groups and virtual links. Using a (virtual) switch on a set $X$ it is possible to construct a representation of the (virtual) braid group on $n$ strands by permutations of $X^n$ (see \cite[Section~2]{FMK}). If $X$ is an algebraic system, then under additional conditions it is possible to construct a representation of the (virtual) braid group by automorphisms of $X$. The Artin representation $\varphi_A:B_n\to {\rm Aut}(F_n)$
 (see \cite[Section~1.4]{Bir}), the Burau representation $\varphi_B:B_n\to {\rm GL}_n(\mathbb{Z}[t,t^{-1}])$ (see  \cite[Section~3]{Bir}) and their extensions to the virtual braid groups $\varphi_A:VB_n\to{\rm Aut}(F_n)$, $\varphi_B:VB_n\to {\rm GL}_n(\mathbb{Z}[t,t^{-1}])$ (see \cite{BarMikNes, Ver}) can be obtained on this way.

Despite the fact that virtual switches can be used for constructing representations of  virtual braid groups, there are representations $VB_n\to{\rm Aut}(G)$, where $G$ is some group, which cannot be obtained using any virtual switch on $G$. For example, the Silver-Williams representation $\varphi_{SW}:VB_n\to {\rm Aut}(F_{n}*\mathbb{Z}^{n+1})$ (see \cite{SilWil}), the Boden-Dies-Gaudreau-Gerlings-Harper-Nicas representation $\varphi_{BD}:VB_n\to {\rm Aut}(F_{n}*\mathbb{Z}^2)$ (see \cite{BDGGHN}),  the Kamada representation $\varphi_K:VB_n\to {\rm Aut}(F_{n}*\mathbb{Z}^{2n})$ (see \cite{BN}) and the representations $\varphi_M:VB_n\to {\rm Aut}(F_n*\mathbb{Z}^{2n+1})$, $\tilde{\varphi}_M:VB_n\to {\rm Aut}(F_n*\mathbb{Z}^n)$ of Bardakov-Mikhalchishina-Neshchadim (see \cite{BarMikNes, BarMikNes2}) cannot be obtained using any virtual switch.

In the present paper we introduce the notion of a (virtual) multi-switch which generalizes the notion of a (virtual) switch. Using (virtual) multi-switches we introduce a general approach how to construct representations of (virtual) braid groups by automorphisms of algebraic systems. As a corollary, we introduce new representations of virtual braid groups which generalize several previously known representations.

The paper is organized as follows. In Section~\ref{prelimin}, we give necessary preliminaries. In particular, we recall the notion of a (virtual) switch, and provide examples and applications of (virtual) switches. In Section~\ref{sectmul}, we introduce the notion of a (virtual) multi-switch and give examples of (virtual) multi-switches. In Section~\ref{reprrr}, we describe a general construction of how a (virtual) multi-switch on an algebraic system $X$ can be used to construct a representation of the (virtual) braid group by automorphisms of $X$ (Theorem~\ref{autrepr} and Theorem~\ref{vautrepr}). As a corollary, we construct a representation $VB_n\to {\rm Aut}(FQ_n*T_n)$
of the virtual braid group $VB_n$ by automorphisms of the free product $FQ_n*T_n$ of the free quandle $FQ_n$ with $n$ generators and the trivial quandle $T_n$ on $n$ elements. Finally, in Section~\ref{seclinrep}, we introduce representations of the (virtual) braid groups by automorphisms of certain infinitely generated abelian groups, which extend the Burau representation $\varphi_B$ and the Gassner representation $\varphi_G$ (Theorem~\ref{glin} and Theorem~\ref{vvvglin}).
\subsection*{Acknowledgement} The authors thank Professor Kauffman for useful discussions. 

\section{Preliminaries}\label{prelimin}
In this section we give necessary preliminaries. We recall the notion of a switch from \cite{FMK}, and show some examples and applications of switches. We use classical notations. If $G$ is a group, and $a,b\in G$, then we denote by $a^b=b^{-1}ab$ the conjugate of $a$ by $b$, and by $a^{-b}=b^{-1}a^{-1}b$ the conjugate of $a^{-1}$ by $b$. If $X$ is a set, then we denote by ${\rm Sym}(X)$ the set of all bijections from $X$ to~$X$. All actions are supposed to be right, i.~e. if $f,g\in {\rm Sym}(X)$, then for $x\in X$ we denote $fg(x)=g(f(x))$.

\subsection{Switches and the Yang-Baxter equation} \textit{A set-theoretical solution of the Yang-Baxter equation} is a pair $(X, S)$, where $X$ is a set, and $S:X^2\to X^2$ is a bijective map such that
\begin{equation}\label{YB}(S\times id)(id \times S)(S\times id)=(id\times S)(S \times id)(id\times S).
\end{equation}
 If $(X,S)$ is a set theoretical solution of the Yang-Baxter equation, then the map $S$ from ${\rm Sym}(X^2)$ is called \textit{a switch} on $X$ (see \cite[Section~2]{FMK}). A switch $S$ is called \textit{involutive} if $S^2=id$. If $X$ is not just a set but an algebraic system: group, module etc, then every switch on $X$ is called \textit{a group switch}, \textit{a module switch} etc. Examples of switches on different algebraic systems follow.
\begin{example}
 Let $X$ be an arbitrary set, and $T(a,b)=(b,a)$ for all $a, b\in X$. The map $T$ is a switch which is called \textit{the twist}.
 It is clear that $T$ is involutive.
\end{example}
\begin{example}
Let $X$ be a group. Then the map $S_A$ defined by $
S_A(a, b) = (a b a^{-1}, a)$ for $a,b\in X$ is a switch which is called \textit{the Artin switch} on a group $X$.
\end{example}
\begin{example}\label{bsw}Let $X$ be a module over an integral domain $R$, and $t$ be an invertible element of $R$. The linear isomorphisms $S_B : X^2 \to X^2$ given by
\begin{align}
\notag S_B(a, b) = ((1 - t) a + t b, a),&&a,b\in X
\end{align}
is a module switch on $X$ which is called \textit{the Burau switch}.
If we look on $X$ as on abelian group (without module structure), then the Burau switch is also a group switch.
\end{example}
In order to introduce the next example of a switch, let us recall that \textit{a quandle} $Q$ is an algebraic system with one binary algebraic operation $(a,b)\mapsto a*b$ which satisfies the following axioms:
\begin{enumerate}
\item $a*a=a$ for all $a\in Q$,
\item the map $I_a:b\mapsto b*a$ is a bijection of $Q$ for all $a\in Q$,
\item $(a*{b})*c=(a*c)*(b*c)$ for all $a,b,c\in Q$.
\end{enumerate}
A quandle $Q$ is called \textit{trivial} if $a*b=a$ for all $a,b\in Q$, the trivial quandle with $n$ elements is denoted by $T_n$. Quandles were introduced in \cite{Joy, Mat} as an invariant for links. For more details about quandles see \cite{Carter, ElhNel, Nos}.
If $(Q,*)$ is a quandle, then for $a,b\in Q$ we denote by $a*^{-1}b=I_b^{-1}(a)$. The following  example gives a quandle switch.
\begin{example}\label{quasw} Let $(X,*)$ be a quandle. Then the map $S_Q(a, b) = (b * a, a)$ for $a, b \in X$ is a quandle switch.
\end{example}
The following example gives a switch on a skew brace  (see, for example, \cite{GuaVen,Nas} for details about skew braces).
\begin{example} Let $X$ be a skew brace with the operations $\oplus,\odot$. Then the map $S:X^2\to X^2$ given by
$$
S(a,b)=\Big(\ominus a\oplus(a\odot b),(\ominus a\oplus(a\odot b))^{-1}\odot a\odot b\Big)
$$
for $a,b\in X$ is a skew brace switch.
\end{example}
For more examples of switches see \cite[Section~2]{FMK}.
\subsection{Switches and representations of braid groups}\label{swrep} Switches can be used for constructing representations of braid groups. Let us recall the definitions. \textit{The braid group $B_n$ on $n$ strands} is the group with generators $\sigma_1,\sigma_2,\dots,\sigma_{n-1}$ and defining relations
\begin{align}
\sigma_i\sigma_{i+1}\sigma_i&=\sigma_{i+1}\sigma_i\sigma_{i+1}&i=1,2,\dots,n-2,\tag{$b_1$}\label{b1}\\
\sigma_i\sigma_j&=\sigma_{j}\sigma_i&|i-j|\geq2.\tag{$b_2$}\label{b2}
\end{align}
There exists a homomorphism $\iota: B_n\to \Sigma_n$ from the braid group $B_n$ onto the symmetric group $\Sigma_n$ on $n$ letters. This homomorphism maps the generator $\sigma_i$ to the transposition $\tau_i=(i,~i+1)$ for $i=1, 2, \dots, n - 1$. The kernel of this homomorphism is called \textit{the pure
braid group on $n$ strands} and is denoted by $P_n$.

\textit{The virtual braid group} $VB_n$ is a group obtained from $B_n$ adding new generators $\rho_1,\rho_2,\dots,\rho_{n-1}$
and additional relations
\begin{align}
\rho_i\rho_{i+1}\rho_i&=\rho_{i+1}\rho_i\rho_{i+1}&i=1,2,\dots,n-2,\tag{$vb_3$}\label{p1}\\
\rho_i\rho_j&=\rho_{j}\rho_i&|i-j|\geq2,\tag{$vb_4$}\label{p2}\\
\rho_i^2&=1&i=1,2,\dots,n-1,\tag{$vb_5$}\label{p3}\\
\rho_{i+1}\sigma_i\rho_{i+1}&=\rho_i\sigma_{i+1}\rho_i&i=1,2,\dots,n-2,\tag{$vb_6$}\label{m1}\\
\sigma_i\rho_j&=\rho_{j}\sigma_i&|i-j|\geq2.\tag{$vb_7$}\label{m2}
\end{align}
It is easy to verify that the elements $\rho_1,\rho_2,\dots,\rho_{n-1}$ generate the symmetric group $\Sigma_n$ in $VB_n$. Also it is known that the elements $\sigma_1,\sigma_2,\dots,\sigma_{n-1}$ generate
the braid group $B_n$ in $VB_n$. The homomorphism $\iota$ can be extended to the homomorphism $VB_n\to \Sigma_n$ by the rule $\iota:\sigma_i, \rho_i\mapsto\tau_i=(i,~i+1)$. The kernel of this homomorphism is called \textit{the virtual pure braid group} and is denoted by $VP_n$.

Let $S\in {\rm Sym}(X^2)$ be a switch on $X$. For $i=1,2,\dots,n-1$ denote by
$$
S_i = (id)^{i-1} \times S \times (id)^{n-i-1}.
$$
From the relations of $B_n$ and equality (\ref{YB}) we see that the map which maps $\sigma_i$ to $S_i$ for $i=1,2,\dots,n-1$ defines a representation of the braid group $B_n$ into the symmetric group ${\rm Sym}(X^n)$.
If $S$ is involutive, then the map $\tau_i\mapsto S_i$ defines a representation of the symmetric group $\Sigma_n$ on $n$ letters into the group ${\rm Sym}(X^n)$.

Under additional conditions switches can provide representations of braid groups by automorphisms of algebraic system. If $X$ is an algebraic system (for example, a quandle, a group, a module etc) generated by elements $x_1,x_2,\dots,x_n$, and $S$ is a switch on $X$
with
$$S(a,b)=(S^l(a,b), S^r(a,b))$$
 for $a,b\in X$, then for $i=1,2,\dots,n-1$ denote by $S_i:\{x_1,x_2,\dots,x_n\}\to X$ the map given by
$$
S_i(x_k)=\begin{cases}S^l(x_k,x_{k+1}),& k=i,\\
S^r(x_k,x_{k+1}),&k=i+1,\\
x_k,&k\neq i,i+1.
\end{cases}
$$
If for $i=1,2,\dots,n-1$ the maps $S_i$ induce automorphisms of $X$, then the map which maps $\sigma_i$ to $S_i$ for $i=1,2,\dots,n-1$ induces a representation $B_n\to {\rm Aut}(X)$. Note that if $X$ is a free in some variety group with the canonical generators $x_1,x_2,\dots,x_n$, then the maps  $S_i$ always induce automorphisms of $X$. The Artin representation $B_n\to {\rm Aut}(F_n)$ (see \cite[Corollary~1.8.3]{Bir}) can be obtained in this way using the Artin switch $S_A$ on $F_n$. The Burau representation $B_n\to {\rm GL}_n\left(\mathbb{Z}[t,t^{-1}]\right)$ (see \cite[Section~3]{Bir}) can be obtained in this way using the Burau switch $S_B$ on $\left(\mathbb{Z}[t,t^{-1}]\right)^n$. It is important here that the algebraic system $X$ has exactly $n$ generators.

Let $S, V\in {\rm Sym}(X^2)$ be a switch and an involutive switch on $X$, respectively. We say that the pair $(S,V)$ is \textit{a virtual switch on $X$} if the equality
\begin{equation}\label{VYB}
 (id \times V) (S \times id)(id \times V)  =   (V \times id)(id \times S)(V \times id)
\end{equation}
holds. If the pair of switches $S,V$ satisfies equality (\ref{VYB}), then we say that this pair is \textit{matched}.
\begin{example} If $X$ is a group, then $(S_A,T)$ is a virtual group switch on $X$.
\end{example}

For a virtual switch $(S, V)$ on $X$ and an integer $n\geq2$ denote by
\begin{align}
\notag S_i = (id)^{i-1} \times S \times (id)^{n-i-1},&&V_i = (id)^{i-1} \times V \times (id)^{n-i-1}
\end{align}
for $i=1,2,\dots,n-1$. From the relations of $VB_n$ and equalities (\ref{YB}), (\ref{VYB}) we see that the maps $\sigma_i \mapsto S_i$, $\rho_i\mapsto V_i$ induce a representation of the virtual  braid group $VB_n$ on $n$ strands into the symmetric group ${\rm Sym}(X^n)$.

If $X$ is an algebraic system generated by elements $x_1,x_2,\dots,x_n$, and $(S,V)$ is a virtual switch on $X$
with
\begin{align}\notag S(a,b)=(S^l(a,b), S^r(a,b)),&& V(a,b)=(V^l(a,b),V^r(a,b))
\end{align}
for $a,b\in X$, then for $i=1,2,\dots,n-1$ denote by $S_i, V_i:\{x_1,x_2,\dots,x_n\}\to X$ the maps given by
$$
S_i(x_k)=\begin{cases}S^l(x_k,x_{k+1}),& k=i,\\
S^r(x_k,x_{k+1}),&k=i+1,\\
x_k,&k\neq i,i+1,
\end{cases}~V_i(x_k)=\begin{cases}V^l(x_k,x_{k+1}),& k=i,\\
V^r(x_k,x_{k+1}),&k=i+1,\\
x_k,&k\neq i,i+1.
\end{cases}
$$
If $S_i,V_i$ induce automorphisms of $X$, then the maps $\sigma_i\mapsto S_i$, $\rho_i\mapsto V_i$ induce a representation $VB_n\to {\rm Aut}(X)$.
\subsection{Biquandles}\label{swinv} Let $S$ be a switch on $X$ such that
$$S(a,b)=(S^l(a,b), S^r(a,b)).$$
For $a,b\in X$ denote by  $S^l(a,b)=b^a$, $S^r(a,b)=a_b$, so, on $X$ we have two binary algebraic operations $(a, b) \mapsto a^b$, $(a, b) \mapsto a_b$, which are called \textit{the up operation} and \textit{the down operation} defined by $S$, respectively. The Yang-Baxter equation for $S$ implies the following equalities
\begin{align}
\label{biqiq} a^{bc}=a^{c_bb^c},&&a_{bc}=a_{c^bb_c},&&
(a^b)_{c^{b_a}}=a_{c^{b_{c^a}}}
\end{align}
for all $a,b,c\in X$. A switch $S$ is called \textit{a biquandle switch} if the following conditions hold.
\begin{enumerate}
\item The maps $f^a,f_a:X\to X$ given by  $f^a(x)=x^a$, $f_a(x)=x_a$ are bijective. We denote by $b^{a^{-1}}=(f^a)^{-1}(b)$, $b_{a^{-1}}=(f_a)^{-1}(b)$.
\item $a^{a^{-1}}=a_{a^{a^{-1}}}$ and $a_{a^{-1}}=a^{a_{a^{-1}}}$ for all $a\in X$.
\end{enumerate}

If $S$ is a biquandle switch on $X$, then the set $X$ with the up and the down operations defined by $S$ is called \textit{a biquandle} and is denoted by $(X,S)$. If $X$ is an arbitrary set, then the twist $T$ on $X$ is a biquandle switch. The biquandle $(X,T)$ is called \textit{the trivial biquandle on $X$}. In the trivial biquandle we have $a=a^b=a_b$ for all $a,b\in X$.

Biquandles were introduced in \cite{FMK} as a tool for constructing invariants for virtual knots and braids. Papers \cite{CSWES, KM, F, HK} give several application of biquandles in knot theory.

\section{Multi-switches and virtual multi-switches}\label{sectmul}
 In this section we consider a special kind of (virtual) switches, called (virtual) multi-switches, which help to provide new representations of (virtual) braid groups.

\subsection{Multi-switches} Let $X$ be a set, and $X_1,X_2,\dots,X_m$ be non-empty subsets of $X$. We say that a map $S:X^2\times X_1^2\times  X_2^2\times \dots\times X_m^2\to X^2\times X_1^2\times X_2^2\times \dots\times X_m^2$ is \textit{an $(m+1)$-switch}, or \textit{a multi-switch on $X$} (if $m$ is not specified) if $S$ is a switch on $X\times X_1\times X_2\times\dots\times X_m$ (we identify the sets $(X\times X_1\times X_2\times\dots\times X_m)^2$ and $X^2\times X_1^2\times X_2^2\times \dots\times X_m^2$, so, for $(m+1)$-tuples $A=(a_0,a_1,\dots,a_m)$, $B=(b_0,b_1,\dots,b_m)$ from $X\times X_1\times X_2\times\dots\times X_m$ we write $S(A,B)=S(a_0,b_0;a_1,b_1;\dots;a_m,b_m)$), such that
$$S(c_0,c_1,\dots,c_m)=(S_0(c_0,c_1,\dots,c_m),S_1(c_1),S_2(c_2),\dots,S_m(c_m))$$
for $c_0\in X^2$, $c_i\in X_i^2$ for $i=1,2,\dots,m$,
where $S_0, S_1,\dots,S_m$ are the maps
\begin{align}
\notag S_0&:X^2 \times X_1^2 \times X_2^2\times \dots \times X_m^2 \to X^2,\\
\notag S_i&:X_i^2  \to X_i^2,~\text{for}~i = 1, 2, \dots, m.
\end{align}
If $S$ is an $(m+1)$-switch on $X$ defined by the maps $S_0,S_1,\dots,S_m$, then we write $S=(S_0,S_1,\dots,S_m)$. Note that for $i=1,2,\dots,m$ the map $S_i$ is a switch on $X_i$.  If $X$ is not just a set but an algebraic system: group, module etc, then every multi-switch on $X$ is called \textit{a group multi-switch}, \textit{a module multi-switch} etc. We do not require here that $X_1, X_2, \dots,X_m$ are subsystems of $X$.
\begin{example}\label{exx7}Every switch on $X$ is a $1$-switch on $X$.
\end{example}
\begin{example}\label{exx8}If $S$ is a switch on $X$, and  $S_i$ is a switch on $X_i\subset X$ for $i=1,2,\dots,m$, then the map $S\times S_1\times S_2\times\dots\times S_m$ is an $(m+1)$-switch on $X$.
\end{example}
The multi-switches from Example~\ref{exx7} and Example~\ref{exx8} are in some sense trivial. The following proposition provides a module $2$-switch which generalizes the Burau switch. If $M$ is a free module over an integral domain $R$, then we can think about any subset $K$ of $R$ as about subset of $M$ (thinking about $K\subset R$ as about $Kx_0\subset M$ for a fixed non-zero element $x_0$ from $M$).
\begin{prop}\label{Burgen}Let $R$ be an integral domain, $X$ be a free module over $R$, and $X_1$ be a subset of the multiplicative group of $R$. Then the map $S_{2B}:X^2\times X_1^2\to X^2\times X_1^2$ given by
\begin{align}
\notag S_{2B}(a, b; x, y) = ((1 - y) a + x b, a; y, x)&&a,b\in X,~x,y\in X_1
\end{align}
is a $2$-switch on $X$.
\end{prop}
\begin{proof} Using direct calculations we see that the map $S_{2B}^{-1}$ is given by
$$S_{2B}^{-1}(a,b;x,y)=(b,y^{-1}a+y^{-1}(x-1)b;y,x),$$
therefore $S_{2B}$ is bijective, and the only moment we have to check is that equality~(\ref{YB}) holds for $S_{2B}$. Denote by $S_1=S_{2B}\times id$, $S_2=id\times S_{2B}$. Using direct calculations for $a,b,c\in X$, $x,y,z\in X_1$ we have
\begin{align}
\notag S_1S_2S_1&(a,b,c;x,y,z)=\\
\notag&=S_2S_1((1-y)a+xb,a,c;y,x,z)\\
\notag&=S_1((1-y)a+xb,(1-z)a+xc,a;y,z,x)\\
\notag&=((1-z)((1-y)a+xb)+y((1-z)a+xc),(1-y)a+xb,a;z,y,x)\\
\notag&=((1-z)a+x((1-z)b+yc), (1-y)a+xb,a;z,y,x),\\
\notag S_2S_1S_2&(a,b,c;x,y,z)=\\
\notag &=S_1S_2(a,(1-z)b+yc,b;x,z,y)\\
\notag &=S_2((1-z)a+x((1-z)b+yc),a,b;z,x,y)\\
\notag&=((1-z)a+x((1-z)b+yc), (1-y)a+xb,a;z,y,x),
\end{align}
i.~e. $S_1S_2S_1=S_2S_1S_2$, and $S_{2B}$ is a $2$-switch on $X$.
\end{proof}
If in Proposition~\ref{Burgen} we look on $X$ as on abelian group, then the module $2$-switch $S_{2B}$ is also a group $2$-switch. Under conditions of Proposition~\ref{Burgen}, let $t$ be a fixed invertible element from $R$, and $X_1=\{t\}$. Then for elements $a,b\in X$ we have $S_{2B}(a,b;t,t)=((1-t)a+tb,a;t,t)$. Looking only to the first two components we obtain the Burau switch. In this sense, the switch $S_{2B}$ generalizes the Burau switch.

Since an $(m+1)$-switch on $X$ is a switch on $X\times X_1\times X_2\times\dots\times X_m$, the notions of the involutive $(m+1)$-switch and the virtual $(m+1)$-switch follow from the same notions for switches.

\subsection{Virtual multi-switches} Usually, if $(S,V)$ is a virtual switch on a set $X$, then $V=T$ is the twist, and it is quite difficult to find a virtual switch $(S,V)$ with $V\neq T$. V.~Manturov \cite{Man2} found a virtual quandle switch $(S,V)$  such that $V$ is not the twist (see also \cite{Bar-1, BB}). Let $Q$ be a quandle, $T_1=\{t\}$ be the trivial quandle with one element, and  $X=Q*T_1$ the free product of $Q$ and $T_1$ (see more about free products of quandles in \cite{BarNas}). Let $S_Q$ be the quandle switch (from Example~\ref{quasw}) on $X$, and $V:X^2\to X^2$ the involutive switch with $V(a,b)= (b *^{-1} t, a * t)$ for $a, b \in X$. Then $(S_Q, V)$ is a virtual switch on $X$ with $V\neq T$. Using this virtual switch Manturov \cite{Man2}  constructed a quandle invariant for virtual links which generalizes the quandle of Kauffman \cite{Kau}.

In this subsection we construct two virtual multi-switches $(S,V)$ with $V\neq T$: the first one is the virtual module $3$-switch which extends the $2$-switch $S_{2B}$ introduced in Proposition~\ref{Burgen}; the second one is the virtual $2$-switch on a biquandle which leads to a  virtual quandle $2$-switch which generalizes the virtual switch of Manturov.

The following proposition provides a virtual  module $3$-switch.
\begin{prop}\label{Burgen3}Let $R$ be an integral domain, $X$ be the free module over $R$, and $X_1,X_2$ be subsets of the multiplicative group of $R$. Then the pair of maps $(S_{3B},V_{3B})$ from $X^2\times X_1^2\times X_2^2$ to itself  given by
\begin{align}
\notag S_{3B}(a,b;x,y;p,q)&=((1 - y) a + x b, a; y, x;q,p),\\
\notag V_{3B}(a,b;x,y;p,q)&=(pb,q^{-1}a;y,x;q,p)
\end{align}
for $a,b\in X$, $x,y\in X_1$, $p,q\in X_2$ is a virtual $3$-switch on $X$.
\end{prop}
\begin{proof} It is clear that the maps $S_{3B}$, $V_{3B}$ are invertible. Since $S_{3B}$ acts as the $2$-switch $S_{2B}$ introduced in Proposition~\ref{Burgen} on the first four arguments, and as the twist $T$ on the last two arguments, i.~e. $S_{3B}=S_{2B}\times T$, the map $S_{3B}$ is a $3$-switch on $X$. So, we need to check that $V_{3B}$ is an involutive $3$-switch on $X$, and that the pair $(S_{3B}, V_{3B})$ is matched. Denote by $S_1=S_{3B}\times id$, $S_2=id\times S_{3B}$, $V_1=V_{3B}\times id$, $V_2=id\times V_{3B}$. For $a,b,c\in X$, $x,y,z\in X_1$, $p,q,r\in X_2$ we have
\begin{align}
\notag V_1V_2V_1(a,b,c;x,y,z;p,q,r)&=V_2V_1(pb,q^{-1}a,c;y,x,z;q,p,r)\\
\notag &=V_1(pb,pc,r^{-1}q^{-1}a;y,z,x;q,r,p)\\
\notag &=(qpc,r^{-1}pb,r^{-1}q^{-1}a;z,y,x;r,q,p),\\
\notag V_2V_1V_2(a,b,c;x,y,z;p,q,r)&=V_1V_2(a,qc,r^{-1}b;x,z,y;p,r,q)\\
\notag &=V_2(pqc,r^{-1}a,r^{-1}b;z,x,y;r,p,q)\\
\notag &=(pqc,pr^{-1}b,q^{-1}r^{-1}a;z,y,x;r,q,p),
\end{align}
therefore $V_{3B}$ is a $3$-switch on $X$. The equality
$$V^2(a,b;x,y;p,q)=V(pb,q^{-1}a;y,x;q,p)=(a,b;x,y;p,q)$$
implies that $V$ is involutive. The equalities
\begin{align}
\notag V_2S_1V_2(a,b,c;x,y,z;p,q,r)&=S_1V_2(a,qc,r^{-1}b;x,z,y;p,r,q)\\
\notag&=V_2((1-z)a+xqc,a,r^{-1}b;z,x,y;r,p,q)\\
\notag&=((1-z)a+xqc,pr^{-1}b,q^{-1}a;z,y,x;r,q,p)\\
\notag V_1S_2V_1(a,b,c;x,y,z;p,q,r)&=S_2V_1(pb,q^{-1}a,c;y,x,z;q,p,r)\\
\notag &=V_1(pb,(1-z)q^{-1}a+xc,q^{-1}a;y,z,x;q,r,p)\\
\notag &=((1-z)a+qxc,pr^{-1}b,q^{-1}a;z,y,x;r,q,p)
\end{align}
imply that $(S_{3B}, V_{3B})$ is a matched pair of $3$-switches. Therefore $(S_{3B}, V_{3B})$ is a virtual $3$-switch on $X$.
\end{proof}
In the proof of Proposition~\ref{Burgen3} we noticed that $S_{3B}=S_{2B}\times T$. In this sense, the virtual $3$-switch $(S_{3B},V_{3B})$ extends the $2$-switch $S_{2B}$. The following proposition provides a  $2$-switch on a biquandle. 
\begin{prop}\label{biqvmsw}Let $X$ be a biquandle, $X_1$ be a trivial subbiquandle  of $X$, and $S,V$ be the maps from $X^2\times X_1^2$ to itself such that
\begin{align}
\notag S(a, b; x, y) = (b^a, a_b; y, x),&& V(a,b;x,y)=(b^{x^{-1}},a^y;y,x)
\end{align}
for $a,b\in X$, $x,y\in X_1$. If for all $a,b\in X$, $x\in X_1$ the equalities
\begin{align}\label{needeq}
 b^{x^{-1}a}=b^{a^xx^{-1}},&& (a_{b^{y^{-1}}})^y=(a^y)_{b}
\end{align}
hold, then $(S,V)$ is a virtual $2$-switch on $X$.
\end{prop}
\begin{proof} Denote by $S_1=S\times id$, $S_2=id\times S$, $V_1=V\times id$, $V_2=id\times V$. Since $S$ acts separately on the first pair of arguments, and separately on the second pair of arguments, the fact that $S$ is a $2$-switch on $X$ follows from the fact that the map $(a,b)\mapsto (b^a,a_b)$ is a switch on $X$, and the map $(x,y)\mapsto (y,x)$ is a switch on $X_1$. Using direct calculations for $a,b,c\in X$, $x,y,z\in X_1$ we have
\begin{align}
\notag V_1V_2V_1(a,b,c;x,y,z)&=(c^{x^{-1}y^{-1}},b^{x^{-1}z},a^{yz};z,y,x),\\
\notag V_1V_1V_2(a,b,c;x,y,z)&=(c^{y^{-1}x^{-1}},b^{zx^{-1}},a^{zy};z,y,x).
\end{align}
These equalities together with equalities (\ref{biqiq}) and the fact that $X_1$ is a trivial biquandle imply that $V_1V_2V_1=V_2V_1V_2$, i.~e. $V$ is a $2$-switch on $X$. The equalities
$$V^2(a,b;x,y)=V(b^{x^{-1}},a^y;y,x)=(a^{yy^{-1}}, b^{x^{-1}x};x,y)=(a,b;x,y)$$
imply that $V$ is involutive, and the only fact we have to prove is that the pair $S,V$ is matched. For $a,b,c\in X$, $x,y,z\in X_1$ we have
\begin{align}
\notag V_2S_1V_2(a,b,c;x,y,z)&=S_1V_2(a,c^{y^{-1}},b^z;x,z,y)\\
\notag &=V_2(c^{y^{-1}a},a_{c^{y^{-1}}},b^z;z,x,y)\\
\notag &=(c^{y^{-1}a},b^{zx^{-1}},(a_{c^{y^{-1}}})^y;z,y,x)\\
\notag V_1S_2V_1(a,b,c;x,y,z)&=S_2V_1(b^{x^{-1}},a^y,c;y,x,z)\\
\notag &=V_1(b^{x^{-1}},c^{a^y},(a^y)_c;y,z,x)\\
\notag &=(c^{a^yy^{-1}},b^{x^{-1}z},(a^y)_c;z,y,x)
\end{align}
which together with equalities (\ref{needeq}) from the conditions of the proposition imply the equality $V_2S_1V_2=V_1S_2V_1$.
\end{proof}
\begin{cor}\label{gman}Let $X$ be a quandle with a trivial subquandle $X_1$, and $S,V$ be the maps from $X^2\times X_1^{2}$ to itself defined by
\begin{align}
\notag S(a,b;x,y)=(b*a,a;y,x),&&V(a,b;x,y)=(b*^{-1}x,a*y;y,x)
\end{align}
for $a,b\in X$, $x,y\in X_1$. Then $(S,V)$ is a virtual $2$-switch on $X$.
\end{cor}
\begin{proof} For $a,b\in X$ denote by $a^b=a*b$, $a_b=a$. Then $X$ with the operations $(a,b)\mapsto a_b$, $(a,b)\mapsto a^b$ is a biquandle, and $X_1$ is a trivial subbiquandle of this biquandle. Equalities (\ref{needeq}) obviously hold in this situation, and the result follows from Proposition~\ref{biqvmsw}.
\end{proof}
\begin{cor}\label{nweq}Let $Q$ be a quandle, $X_1$ be a trivial quandle, and $X=Q*X_1$ be the free product of $Q$ and $X_1$. Let $S_{2Q},V_{2Q}:X^2\times X_1^{2}\to X^2\times X_1^{2}$ be the maps defined by
\begin{align}
\notag S_{2Q}(a,b;x,y)=(b*a,a;y,x),&&V_{2Q}(a,b;x,y)=(b*^{-1}x,a*y;y,x)
\end{align}
for $a,b\in X$, $x,y\in X_1$. Then $(S_{2Q},V_{2Q})$ is a virtual $2$-switch on $X$.
\end{cor}
If in Corollary~\ref{nweq} we take $X_1$ a trivial quandle with only one element, then looking to the first two components of the maps $S_{2Q},V_{2Q}$ we obtain the switch of Manturov described at the beginning of this section.

\section{Multi-switches and representations of braid groups}\label{reprrr}
In Section~\ref{swrep} we noticed that a (virtual) switch on an algebraic system $X$ with exactly $n$ generators can be used to construct a representation of the (virtual) braid group on $n$ strands by automorphisms of $X$. However, there are representations $VB_n\to {\rm Aut}(G)$ (where $G$ is some group) which cannot be defined using procedure described in Section~\ref{swrep} for any virtual switch $(S,V)$ on $G$. For example, let $F_n$ be the free group with the free generators $x_1, x_2, \dots, x_n$, $\mathbb{Z}^{2n+1}$ be the free abelian group with canonical generators $u_1, u_2, \dots, u_n$, $v_0, v_1, \dots, v_n$, and $F_{n,2n+1} = F_n * \mathbb{Z}^{2n+1}$ be the free product of $F_n$ and $\mathbb{Z}^{2n+1}$, then the representation $\varphi_M:VB_n\to {\rm Aut}(F_{n,2n+1})$ introduced in  \cite{BarMikNes, BarMikNes2} which acts on the generators of
$F_{n,2n+1}$ in the following way
\begin{align}\label{barmnes}
\varphi_M(\sigma_i):
\begin{cases}
x_i \mapsto  x_i x_{i+1}^{u_i} x_i^{- v_0 u_{i+1}},&\\
x_{i+1} \mapsto x_i^{v_0},&\\
u_i\mapsto u_{i+1},&\\
u_{i+1}\mapsto u_i,&\\
v_i\mapsto v_{i+1},&\\
v_{i+1}\mapsto v_i,&\\
\end{cases}&&\varphi_M(\rho_i):
\begin{cases}
x_i \mapsto  x_{i+1}^{v_i^{-1}},&\\
x_{i+1} \mapsto x_i^{v_{i+1}},&\\
u_i\mapsto u_{i+1},&\\
u_{i+1}\mapsto u_i,&\\
v_i\mapsto v_{i+1},&\\
v_{i+1}\mapsto v_i,&\\
\end{cases}
\end{align}
(hereinafter we write only non-trivial actions on the generators, assuming that all other generators are fixed) cannot be defined using procedure described in Section~\ref{swrep} for any virtual switch $(S,V)$ on $F_{n,2n+1}$ (since $F_{n,2n+1}$ has $3n+1$ generators). 

The representation $\tilde{\varphi}_M:VB_n\to {\rm Aut}(F_{n,n})$, where $F_{n,n}=F_n*\mathbb{Z}^n$ is the free product of the free group $F_n$ generated by $x_1,x_2,\dots,x_n$, and the free abelian group $\mathbb{Z}^n$ generated by $y_1,y_2,\dots,y_n$ introduced in \cite{BarMikNes} which acts on the generators of $F_{n,n}$ by the rule
\begin{align}\label{uvartin}
\tilde{\varphi}_M(\sigma_i):\begin{cases}x_i\mapsto x_ix_{i+1}x_i^{-1},\\
x_{i+1} \mapsto x_i,\\
y_i \mapsto y_{i+1},\\
y_{i+1}\mapsto y_i.
\end{cases}&&
\tilde{\varphi}_M(\rho_i):\begin{cases}x_i \mapsto y_ix_{i+1}y_i^{-1},\\
x_{i+1} \mapsto y_{i+1}^{-1}x_iy_{i+1},\\
y_i \mapsto y_{i+1},\\
y_{i+1}\mapsto y_i.
\end{cases}
\end{align}
gives another example of the representation of $VB_n$ which cannot be defined using procedure described in Section~\ref{swrep} for any virtual switch $(S,V)$ on $F_{n,n}$. Note that the representations $\varphi_M$ and $\tilde{\varphi}_M$ given in (\ref{barmnes}), (\ref{uvartin}) have the same kernel \cite{BarMikNes}.

In this section we describe a general construction of how a (virtual) multi-switch on an algebraic system $X$ can be used to construct a representation of the (virtual) braid group on $n$ strands by automorphisms of $X$. We do not require that $X$ has exactly $n$ generators, so, a lot of known representations of virtual braid groups by automorphisms of groups (in particular, representations $\varphi_M$, $\tilde{\varphi}_M$ given by formulas (\ref{barmnes}), (\ref{uvartin}), respectively) can be constructed using the procedure which we will introduce in this section.

Since a virtual $(m+1)$-switch $(S,V)$ on an algebraic system $X$ is a switch on $X\times X_1\times X_2\times\dots\times X_m$, from the first glance it can seem that one can apply the procedure from Section~\ref{swrep} in order to construct a representation of the virtual braid group by automorphisms of $X$. However, it is not true due to the fact that $(S,V)$ is a virtual switch on  $X\times X_1\times X_2\times\dots\times X_m$ (but not on $X$), and the set  $X\times X_1\times X_2\times\dots\times X_m$ doesn't necessarily have a structure of algebraic system.
\subsection{General construction}\label{gencon} Let $X$ be an algebraic system, and $X_0,X_1,\dots, X_m$ be subsets of $X$  such that
\begin{enumerate}
\item   for $i=0,1,\dots,m$ the set $X_i$ contains elements $x^{i}_{1},x^{i}_{2},\dots,x^{i}_{n}$,
\item $\{x^{i}_{1}, x^{i}_2,\dots,x^{i}_{n}\}\cap\{x^{j}_{1}, x^{j}_2,\dots,x^{j}_{n}\}=\varnothing$ for $i\neq j$,
\item the set of elements $\{x^{i}_j~|~i=0,1,\dots,m,j=1,2,\dots,n\}$ generates $X$.
\end{enumerate}
Let $S=(S_0,S_1,\dots,S_m)$ be an $(m+1)$-switch on $X$ such that
\begin{align}
\notag S_0=(S_0^l,S_0^r)&:X^2 \times X_1^2 \times X_2^2 \times \dots \times X_m^2 \to X^2,\\
\notag S_i=(S_i^l,S_i^r)&:X_i^2  \to X_i^2,~\text{for}~i = 1, 2, \dots, m,
\end{align}
and for $i=0,1,\dots,m$ the images of maps $S_i^l,S_i^r$  are words over its arguments in terms of operations of $X$. For $j=1,2,\dots,n-1$ denote by $R_j$ the following map from $X_0\cup X_1\cup \dots \cup X_m$ to $X$
\begin{align}
\label{fj} R_j :&\begin{cases}
x^0_{j} \mapsto S_0^l(x^0_{j}, x^0_{j+1}, x^1_{j}, x^1_{j+1}, \dots, x^m_{j}, x^m_{j+1}),\\
x^0_{j+1} \mapsto S_0^r(x^0_{j}, x^0_{j+1}, x^1_{j}, x^1_{j+1}, \dots, x^m_{j}, x^m_{j+1}),\\
x^1_{j} \mapsto S_1^l(x^1_{j}, x^1_{j+1}),\\
x^1_{j+1} \mapsto S_1^r(x^1_{j}, x^1_{j+1}), \\
~~\vdots\\
x^m_{j} \mapsto S_m^l(x^m_{j}, x^m_{j+1}),\\
x^m_{j+1} \mapsto S_m^r(x^m_{j}, x^m_{j+1}),
\end{cases}
\end{align}
where all generators which are not explicitly written in $R_j$ are fixed, i.~e. $R_j(x_k^i)=x_k^i$ for $k\notin \{j,j+1\}$, $i=0,1,\dots,m$, and assume that $R_j$ is well defined: since the elements $x^{i}_{1},x^{i}_{2},\dots,x^{i}_{n}$ are not necessary all different, some of these elements can coincide. The fact that $R_j$ is well defined means that the images of equal elements are equal. For example, if $x_j^i=x_{j+1}^i$, then we assume that
$$S_i^l(x_j^i,x_{j+1}^i)=R_j(x_j^i)=R_j(x_{j+1}^i)=S_i^r(x_j^i,x_{j+1}^i).$$
If for $j=1,2,\dots,n-1$ the maps $R_j$, induce automorphisms of $X$, then we say that $S$ is \textit{an automorphic multi-switch} (shortly, AMS) or an automorphic $(m+1)$-switch on $X$ with respect to the set of generators $\{x^i_{j}~|~i=0,1,\dots,m,j=1,2,\dots,n\}$. An $(m+1)$-switch can be AMS with respect to one generating set of $X$, but not AMS with respect to another generating set of $X$.
\begin{theorem}\label{autrepr}
Let $X$ be an algebraic system, and $S$ be an AMS on $X$
with respect to the set of generators $\{x^i_{j}~|~i=0,1,\dots,m,j=1,2,\dots,n\}$. Then the map
$$
\varphi_{S} : B_n \to {\rm Aut}(X)
$$
which is defined on the generators of $B_n$ as
\begin{align}
\notag\varphi_{S}(\sigma_j) = R_j,&&{\text for}~j = 1, 2, \dots, n-1,
\end{align}
where $R_j$ is defined by equality (\ref{fj}), is a representation of $B_n$.
\end{theorem}
\begin{proof} We need to check that the automorphisms $R_1,R_2,\dots,R_{n-1}$ satisfy defining relations  of $B_n$. Relations (\ref{b1}) follow from the fact that $S$ is an $(m+1)$-switch on $X$. In order to check relations (\ref{b2}) let $1\leq j,k\leq n$ be such that $|j-k|\geq 2$. For the map $R_j$ we have
$$R_j(x_s^i)=\begin{cases}u_s^i,& s=j,j+1;i=0,1,\dots,m,\\
x_s^i, &s\neq j,j+1;i=0,1,\dots,m,
\end{cases}$$
where $u_s^i$ is a word over $\{x^i_{j},x^i_{j+1}~|~i=0,1,\dots,m\}$ in terms of operations of $X$. Since $R_k$ fixes elements $\{x^i_{j},x^i_{j+1}~|~i=0,1,\dots,m\}$, we have $R_k(u_s^i)=u_s^{i}$ for $s\in\{j,j+1\}$, $i=0,1,\dots,m$. Therefore
$$R_jR_k(x_s^i)=\begin{cases}u_s^i,& s=j,j+1;i=0,1,\dots,m,\\
v_s^i,& s=k,k+1;i=0,1,\dots,m,\\
x_s^i, &s\notin \{j,j+1,k,k+1\};i=0,1,\dots,m,
\end{cases}$$
where $v_s^i$ is a word over $\{x^i_{k},x^i_{k+1}~|~i=0,1,\dots,m\}$ in terms of operations of $X$. The last equality can be rewritten in the form
$$R_jR_k(x_s^i)=\begin{cases}R_j(x_s^i),& s=j,j+1;i=0,1,\dots,m,\\
R_k(x_s^i),& s=k,k+1;i=0,1,\dots,m,\\
x_s^i, &s\notin\{j,j+1,k,k+1\}; i=0,1,\dots,m.
\end{cases}$$
In a similar way we can prove that $R_kR_j(x_s^i)$ can be written in the same form, therefore $R_jR_k=R_kR_j$ for $|j-k|\geq 2$.
\end{proof}
Similar to Theorem~\ref{autrepr} result can be formulated for virtual $(m+1)$-switches and representations of virtual braid groups.  Let $S=(S_0,S_1,\dots,S_m)$, $V=(V_0,V_1,\dots,V_m)$ be a virtual $(m+1)$-switch on $X$ (i.~e. the pair $(S, V)$ is matched) such that
\begin{align}
\notag S_0=(S_0^l,S_0^r), V_0=(V_0^l,V_0^r)&:X^2 \times X_1^2 \times X_2^2 \times \dots \times X_m^2 \to X^2,\\
\notag S_i=(S_i^l,S_i^r), V_i=(V_i^l,V_i^r)&:X_i^2  \to X_i^2,~\text{for}~i = 1, 2, \dots, m,
\end{align}
and for $i=0,1,\dots,m$ the images of maps $S_i^l,S_i^r, V_i^l,V_i^r$  are words over its arguments in terms of operations of $X$. 
For $j=1,2,\dots,n-1$ denote by $R_j$ the map from $X_0\cup X_1\cup \dots \cup X_m$ to $X$ given by equality (\ref{fj}), and by $G_j$ the following map from $X_0\cup X_1\cup \dots \cup X_m$ to $X$.
\begin{align}
\label{gj} G_j :&\begin{cases}
x^0_{j} \mapsto V_0^l(x^0_{j}, x^0_{j+1}, x^1_{j}, x^1_{j+1}, \dots, x^m_{j}, x^m_{j+1}),\\
x^0_{j+1} \mapsto V_0^r(x^0_{j}, x^0_{j+1}, x^1_{j}, x^1_{j+1}, \dots, x^m_{j}, x^m_{j+1}),\\
x^1_{j} \mapsto V_1^l(x^1_{j}, x^1_{j+1}),\\
x^1_{j+1} \mapsto V_1^r(x^1_{j}, x^1_{j+1}), \\
~~\vdots\\
x^m_{j} \mapsto V_m^l(x^m_{j}, x^m_{j+1}),\\
x^m_{j+1} \mapsto V_m^r(x^m_{j}, x^m_{j+1}),
\end{cases}
\end{align}
and suppose that $R_j, G_j$ are well defined. If for $j=1,2,\dots,n-1$ the maps $R_j$, $G_j$ induce automorphisms of $X$, then we say that $(S,V)$ is \textit{an automorphic virtual multi-switch} (shortly, AVMS) or an automorphic virtual $(m+1)$-switch on $X$ with respect to the set of generators $\{x^i_{j}~|~i=0,1,\dots,m,j=1,2,\dots,n\}$.
\begin{theorem}\label{vautrepr}
Let $(S, V)$ be an AVMS on $X$
with respect to the set of generators $\{x^i_{j}~|~i=0,1,\dots,m,j=1,2,\dots,n\}$. Then the map
$$
\varphi_{S,V} : VB_n \to {\rm Aut}(X)
$$
which is defined on the generators of $VB_n$ as
\begin{align}
\notag\varphi_{S,V}(\sigma_j) = R_j,&&\varphi_{S,V}(\rho_j) = G_j,&&{\text for}~j = 1, 2, \dots, n-1,
\end{align}
where $R_j$, $G_j$ are defined by equalities (\ref{fj}), (\ref{gj}),
is a representation of $VB_n$.
\end{theorem}
\begin{proof} The proof repeats the proof of Theorem~\ref{autrepr} adding a few details.
\end{proof}
A lot of known representations of the virtual braid group $VB_n$ by automorphisms of groups can be obtained using Theorem~\ref{vautrepr}. For example, the representation $\varphi_M$ given in equalities (\ref{barmnes}) is obtained in the following way: denote by $X=F_{n,2n+1}=F_n*\mathbb{Z}^{2n+1}$, $X_0=F_n=\langle x_1,x_2,\dots,x_n\rangle$, $X_1=\langle u_1,u_2,\dots,u_n\rangle$, $X_2=\langle v_1,v_2,\dots,v_n\rangle$, $X_3=\langle v_0\rangle$, and for $j=1,2,\dots,n$ denote by $x^0_j=x_j$, $x^1_j=u_j$, $x^2_j=v_j$, $x^3_j=v_0$. Then the representation $\varphi_M$ is the representation $\varphi_{S,V}$ for the virtual $4$-switch $(S,V)$ on $F_{n,2n+1}$ given by
\begin{align}
\notag S(a,b;x,y;p,q;r,s)&=(a b^x a^{-ry}, a^{s};y,x;q,p;s,r),\\ \notag V(a,b;x,y;p,q;r,s)&=(b^{p^{-1}}, a^q;y,x;q,p;s,r),
\end{align}
for $a,b\in X$, $x,y\in X_1$, $p,q\in X_2$, $r,s\in X_3$ (the fact that $(S,V)$ is a virtual $4$-switch on $F_{n,2n+1}$ follows from the fact that $\varphi_M$ is a representation of the virtual braid group).

The representation $\tilde{\varphi}_M:VB_n\to {\rm Aut}(F_{n}*\mathbb{Z}^n)$ given in equality (\ref{quvartin}), the Silver-Williams representation $\varphi_{SW}:VB_n\to {\rm Aut}(F_{n}*\mathbb{Z}^{n+1})$  (see \cite{SilWil}), the Boden-Dies-Gaudreau-Gerlings-Harper-Nicas representation $\varphi_{BD}:VB_n\to {\rm Aut}(F_{n}*\mathbb{Z}^2)$ (see~\cite{BDGGHN}), and the Kamada representation $\varphi_K:VB_n\to {\rm Aut}(F_{n}*\mathbb{Z}^{2n})$ (see~\cite{BN}) can be obtained using Theorem~\ref{vautrepr} in a similar to $\varphi_M$ way.

In the following subsection using Theorem~\ref{vautrepr} we construct a representations of the virtual braid groups by automorphisms of quandles.

\subsection{Representations by automorphisms of quandles}\label{quandlerepr} Let $X_0=FQ_n$ be the free quandle on $n$ generators $x^0_{1},x^0_2,\dots,x^0_{n}$, $X_1=T_n$ be the trivial quandle on $n$ elements $x^1_{1},x^1_2\dots,x^1_{n}$, and $X=X_0*X_1$ be the free product of $X_0$ and $X_1$. From Corollary~\ref{nweq} we know that the maps $S_{2Q},V_{2Q}:X^2\times X_1^{2}\to X^2\times X_1^{2}$ defined by
\begin{align}
\notag S_{2Q}(a,b;x,y)=(b*a,a;y,x),&&V_{2Q}(a,b;x,y)=(b*^{-1}x,a*y;y,x)
\end{align}
for $a,b\in X$, $x,y\in X_1$ form a virtual $2$-switch on $X$. Due to the fact that $X=X_0*X_1$ is a free product of the free quandle and the trivial quandle, it is easy to see that $(S,V)$ is an automorphic virtual $2$-switch on $X$ with respect to the set of generators $x^0_{1},x^0_2,\dots,x^0_{n}$, $x^1_{1},x^1_2,\dots,x^1_{n}$. So, Theorem~\ref{vautrepr} and Corollary~{\ref{nweq}} imply the following result.
\begin{theorem}\label{thrr}
 Let $FQ_{n}$ be the free quandle on the set of generators $\{x_1,x_2,\dots,x_n\}$ and $T_n=\{y_1,y_2,\dots,y_n\}$ be the trivial quandle. Then the map $\varphi_{2Q}$ given by
\begin{align}\label{quvartin}
\varphi_{2Q}(\sigma_i):\begin{cases}x_i\mapsto x_{i+1}*x_i,\\
x_{i+1} \mapsto x_i,\\
y_i \mapsto y_{i+1},\\
y_{i+1}\mapsto y_i,\\
\end{cases}&&
\varphi_{2Q}(\rho_i):\begin{cases}x_i \mapsto x_{i+1}*^{-1}y_i,\\
x_{i+1} \mapsto x_i*y_{i+1},\\
y_i \mapsto y_{i+1},\\
y_{i+1}\mapsto y_i,\\
\end{cases}
\end{align}
induces a homomorphism $VB_n\to{\rm Aut}(FQ_{n}*T_n)$.
\end{theorem}
\begin{proof}Denote by $x^0_{i}=x_i$, $x^1_{i}=y_i$ for $i=1,2,\dots,n$, $X_0=\langle x^0_{1},x^0_2,\dots,x^0_{n}\rangle=FQ_n$, $X_1=\langle x^1_{0},x^1_2,\dots,x^1_{n}\rangle=T_n$, $X=X_0*X_1$, and $(S_{2Q},V_{2Q})$ the virtual $2$-switch on $X$ defined in Corollary~\ref{nweq}. Then the map $\varphi_{2Q}$ described in the formulation of the theorem is the map $\varphi_{S_{2Q},V_{2Q}}$ which is a homomorphism  $VB_n\to {\rm Aut}(X)$ by Theorem~\ref{vautrepr}.
\end{proof}
The representation $\varphi_{2Q}$ given in Theorem~\ref{thrr} generalizes the representation $\tilde{\varphi}_M$ given by formulas (\ref{uvartin}). Note that $FQ_n*T_n$ has $2n$ generators, so, the representation given in Theorem~\ref{thrr} cannot be obtained using procedure described in Section~\ref{swrep}. The following proposition gives a representation of the virtual braid group $VB_n$ by automorphisms of the free quandle $FQ_{n+1}$ with $n+1$ generators.
\begin{prop}\label{ezthrr}
 Let $FQ_{n+1}$ be the free quandle with generators $\{x_1,x_2,\dots,x_n,y\}$. Then the map $\varphi$ given by
\begin{align}\label{automs}
\varphi(\sigma_i):\begin{cases}x_i\mapsto x_{i+1}*x_i,\\
x_{i+1} \mapsto x_i,
\end{cases}&&\varphi(\rho_i):\begin{cases}x_i \mapsto x_{i+1}*^{-1}y,\\
x_{i+1} \mapsto x_i*y
\end{cases}
\end{align}
induces a homomorphism $VB_n\to{\rm Aut}(FQ_{n+1})$.
\end{prop}
\begin{proof}Denote by $x^0_{i}=x_i$, $x^1_{i}=y$ for $i=1,2,\dots,n$, $X_0=\langle x^0_{1},x^0_2,\dots,x^0_{n}\rangle=FQ_n$, $X_1=\langle x^1_{0},x^1_1\dots,x^1_{n}\rangle=T_1=\{y\}$, $X=X_0*X_1=FQ_{n+1}$, $(S_{2Q},V_{2Q})$ the virtual $2$-switch on $X$ defined in Corollary~\ref{nweq}. Since $FQ_{n+1}$ is the free quandle, and the maps $\varphi(\sigma_i), \varphi(\rho_i)$ from equality (\ref{automs}) are invertible, $\varphi(\sigma_i), \varphi(\rho_i)$ induce automorphisms of $FQ_{n+1}$, therefore $(S_{2Q},V_{2Q})$ is AVMS with respect to the set of generators $\{x^0_1,x^0_1,\dots,x^0_n,x^1_1,x^1_2,\dots,x^1_n\}$ (the maps $\varphi(\sigma_i), \varphi(\rho_i)$ from (\ref{automs}) are exactly the maps $R_i$, $G_i$ from (\ref{fj}), (\ref{gj}) for $(S_{2Q},V_{2Q})$). Then the map $\varphi$ described in the formulation of the proposition is the map $\varphi_{S_{2Q},V_{2Q}}$ which is a representation $VB_n\to {\rm Aut}(X)$ by Theorem~\ref{vautrepr}.
\end{proof}
The representation given in Proposition~\ref{ezthrr} generalizes the representation $\psi$ introduced in \cite{Bar-1}.

\section{Multi-switches and linear representations of braid groups}\label{seclinrep}
The question of whether the group $B_n$ is linear was a long standing problem. In 1936 Burau  constructed a linear representation $\varphi_B:B_n\to{\rm GL}_n(\mathbb{Z}[t,t^{-1}])$ (see, for example, \cite[Section~3]{Bir}) which is given on the generators of $B_n$ by the following equality
$$\varphi_B(\sigma_i)={\rm I}_{i-1}\oplus
\begin{pmatrix}1-t&t\\
1&0\end{pmatrix}\oplus {\rm I}_{n-i-1}.$$
This representation can be obtained from the Burau switch (Example~\ref{bsw}) using procedure described in Section~\ref{swrep}. For $n=2,3$ the Burau representation in known to be faithful, for $n\geq 5$ the Burau representation is known to have non-trivial kernel \cite{Big}, for $n=4$ the question on faithfulness of the Burau representation is still open.

The Burau representation was generalized in 1961 by Gassner \cite{Gas} to a related representation of the pure braid group $\varphi_G:P_n\to {\rm GL}_n(\mathbb{Z}[t_1^{\pm1},t_2^{\pm1},\dots,t_n^{\pm1}])$. Let us recall the definition. Let $F_n$ be the free group with the free generators $x_1,x_2,\dots,x_n$. For an element $w\in F_n$ and the generator $x_i$ denote by $\partial w/\partial x_i$ the Fox derivative of $w$ by $x_i$ (see \cite[Section~3.1]{Bir}). The Artin representation $\varphi_A:B_n\to {\rm Aut}(F_n)$ of the braid group $B_n$ by automorphisms of $F_n$ is defined on the generators $\sigma_1,\sigma_2,\dots,\sigma_{n-1}$ of $B_n$ by the following rule
$$
\varphi_A(\sigma_i) :
\begin{cases}
  x_i \mapsto  x_i x_{i+1} x_i^{-1},\\
  x_{i+1} \mapsto x_i,\\
\end{cases}
$$
for $i=1,2,\dots,n-1$. Denote by $~^{ab}:F_n\to \mathbb{Z}^n$ the abelianization map, and let $x_i^{ab}=t_i$ for $i=1,2,\dots,n$. Denote by the same symbol $~^{ab}$ the induced homomorphism $\mathbb{Z}F_n\to \mathbb{Z}[t_1^{\pm1},t_2^{\pm1},\dots,t_n^{\pm1}]$. The homomorphism $\varphi_G:P_n\to {\rm GL}_n(\mathbb{Z}[t_1^{\pm1},t_2^{\pm1},\dots,t_n^{\pm1}])$ given on the generators
$$a_{i,j} = \sigma_{j-1} \sigma_{j-2} \dots \sigma_{i+1} \sigma_{i}^2 \sigma_{i+1}^{-1} \dots \sigma_{j-2}^{-1} \sigma_{j-1}^{-1},~~1\leq i<j\leq n$$
of $P_n$ by the rule
$$\varphi_G(a_{i,j})=\left(\left(\frac{\partial\varphi_A(a_{i,j})(x_r)}{\partial x_s}\right)^{ab}\right)_{r,s}$$
 is called the Gassner representation $\varphi_G:P_n\to {\rm GL}_n(\mathbb{Z}[t_1^{\pm1},t_2^{\pm1},\dots,t_n^{\pm1}])$. If we put $t_1=t_2=\dots=t_n=t$, then we obtain the restriction of the Burau representation to $P_n$. Denote by $e_1,e_2,\dots,e_n$ the natural basis in $(\mathbb{Z}[t_1^{\pm1},t_2^{\pm1},\dots,t_n^{\pm1}])^n$, then the Gassner representation is given by the formulas
\begin{align}\label{gassner}
\varphi_G(a_{i,j})e_k&=\begin{cases}e_k,&k<i,\\
(1-t_i+t_it_j)e_i
+t_i(1-t_i)e_j,&k=i\\
(1-t_{k})(1-t_j)e_i+e_k+(1-t_{k})(t_i-1)e_j,&i<k<j,\\
(1-t_j)e_i+t_ie_j,&k=j\\
e_k,&k>j,
\end{cases}
\end{align}
(see, for example, \cite{Knu} and \cite[Section~3.2]{Bir}).

It is not known whether the Gassner representation is faithful for $n>3$. Some results concerning the faithfulness of the Gassner representation can be found, for example, in \cite{Abd, Knu}.

The final answer to the question about the linearity of the braid groups was given independently by Bigelow \cite{Big2} and Krammer \cite{Kra} who proved that the representation $\varphi_{LKB}:B_n\to{\rm GL}_{n(n-1)/2}(\mathbb{Z}[q^{\pm1}, t^{\pm 1}])$ introduced by Lawrence in \cite{Law} is faithful. The representation $\varphi_{LKB}$ is called now \textit{the Lawrence-Krammer-Bigelow representation}.

The situation with virtual braid groups is more complicated. It is not difficult to prove that the braid group $VB_3$ on $3$ strands is linear.
\begin{prop}
 $VB_3$ is linear.
\end{prop}
\begin{proof}
 The kernel of of the endomorphism $VB_3 \to \Sigma_3$, which is defined on the generators $\sigma_1,\sigma_2,\rho_1,\rho_2$ of $VB_3$ by the rule
\begin{align}
\sigma_1,\sigma_2 \mapsto 1,&&\rho_1 \mapsto (1,2),&&\rho_2\mapsto (2,3)
\end{align}
is a normal subgroup $H_3$ of index $6$ in $VB_3$. So, $VB_3$ is linear if and only if $H_3$ is linear and it is enough to prove that $H_3$ is linear. The presentation of $H_3$ by generators and defining relations was found in master thesis of Rabenda \cite{R} (see also \cite{BB}). This group is generated by the elements
$$
x_{12} = \sigma_1, ~~x_{23} = \sigma_2,~~ x_{21} = \rho_1 \sigma_1 \rho_1,~~ x_{32} = \rho_2 \sigma_2 \rho_2,~~ x_{13} = \rho_2 \sigma_1 \rho_2,~~ x_{31} = \rho_2 \rho_1 \sigma_1 \rho_1 \rho_2,
$$
and is defined by the relations
\begin{equation} \label{eq41}
x_{i,k} \,  x_{k,j} \,  x_{i,k} =  x_{k,j} \,  x_{i,k} \, x_{k,j},
\end{equation}
where  distinct letters stand for distinct indices. From this presentation follows (see \cite[Remark~20]{BB}) that $H_3 = G_1 * G_2$ is a free product of groups
$$
G_1 = \langle x_{12}, x_{23}, x_{31} \rangle,~~~G_2 = \langle x_{13}, x_{32}, x_{21} \rangle.
$$
The groups $G_1,G_2$ are isomorphic to the circular braid group on $3$ strands \cite{A}, which can be embedded into the braid group $B_4$ on $4$ strands \cite{KenPei}. Since $B_4$ is linear, the groups $G_1, G_2$ are linear.  Since the free product of two linear groups is linear \cite{Nis}, we conclude that $H_3$ is linear, and therefore $VB_3$ is linear.
\end{proof}

In \cite{Bar-2} the Gassner representation was extended to the group $Cb_n$ of basis-conjugating automorphisms, which is isomorphic to the pure welded braid group $WP_n$. This extension is not faithful for $n \geq 2$. In \cite{Ver} the Burau representation was extended to the welded braid group  $WP_n$. The question about the linearity of $VP_n$ for $n\geq4$ is formulated in \cite[Problem~19.7(b)]{kt}. In general, it is not known if the virtual braid groups  are linear or not. Also, there are no good linear representation of these groups.

In Section~\ref{gencon} we introduced a general construction how a (virtual) multi-switch on an algebraic system $X$ with finitely many generators can be used for constructing a representation of the (virtual) braid group by automorphisms of $X$ (Theorem~\ref{autrepr} and Theorem~\ref{vautrepr}). Sometimes using the same approach it is possible to construct representations of (virtual) braid groups by automorphisms of infinitely generated algebraic systems. In this section we introduce representations of (virtual) braid groups by automorphisms of some infinitely generated abelian groups. These representations lead to the linear representations of pure (virtual) braid groups which are strongly related with the Burau representation $\varphi_B$ and the Gassner representation $\varphi_G$

\subsection{Representations of $B_n$ and $P_n$} Let $n\geq 2$ be a positive integer, and $M$ be the free left module with the free basis $e_1, e_2, \dots, e_n$ over the ring $K = \mathbb{Z}[t_1^{\pm 1},t_2^{\pm 1}, \dots, t_n^{\pm 1}]$. Denote by $X$ the additive group of $M$, by $X_0$ the subgroup of $X$ generated by $e_1,e_2,\dots,e_n$, and by $X_1=\{t_1,t_2,\dots,t_n\}$ (similarly to Proposition~\ref{Burgen} we can assume that $X_1\subset X$). Let $S_{2B}$ be the $2$-switch on $X$ from Proposition~\ref{Burgen}
$$
S_{2B}(a, b; x, y) = ((1 - y) a + x b, a; y, x)
$$
for $a,b\in X$, $x,y\in X_1$. For $j=1,2,\dots,n-1$ denote by $R_j$ the following map from $\{e_1,e_2,\dots,e_n,t_1,t_2,\dots,t_n\}$ to $X$
\begin{equation}\label{trylin}
R_j :\begin{cases}
 e_j \mapsto  (1 - t_{j+1}) e_j + t_j e_{j+1}, \\
  e_{j+1} \mapsto e_j,  \\
t_j \mapsto  t_{j+1}, \\
t_{j+1} \mapsto t_j.
\end{cases}
\end{equation}
Note that the map $R_j$ from (\ref{trylin}) is the same as the map $R_j$ from equality (\ref{fj}) if we denote by $x^0_i=e_i$, $x^1_i=t_i$ for $i=1,2,\dots,n$. The map $R_j$ induces an automorphism of $X$ by the rule
\begin{equation}\label{ext}R_j\left(\sum_{k=1}^n\alpha_ke_k\right)=\sum_{k=1}^nR_j(\alpha_k)R_j(e_k),
\end{equation}
where $\alpha_1,\alpha_2,\dots,\alpha_n\in K$, and $R_j(\alpha_k)$ denotes the image of $\alpha_k$ under the map $K\to K$ given by permutation $(t_j,~t_{j+1})$ induced by $R_j$. Denote by $\varphi_{2B}$ the map from the set of generators $\{\sigma_1,\sigma_2,\dots,\sigma_{n-1}\}$ of $B_n$ to ${\rm Aut}(X)$ which maps $\sigma_j$ to $R_j$ for $j=1,2,\dots,n-1$.
\begin{theorem}\label{glin}The map $\varphi_{2B} : \{\sigma_1,\sigma_2,\dots,\sigma_{n-1}\} \to {\rm Aut}(X)$ induces a representation of the braid group $B_n$. The restriction of $\varphi_{2B}$ to the pure braid group $P_n$ is a linear representation $\varphi_{2B}:P_n \to {\rm GL}_n(K)$, which coincides with the Gassner representation~$\varphi_G$.
\end{theorem}
\begin{proof} In order to prove that the map $\varphi_{2B}$ induces a representation of the braid group $B_n$ it is necessary to check that the maps $R_1,R_2,\dots,R_{n-1}$ satisfy the defining relations of $B_n$. It can be checked by (a bit massive but not difficult) direct calculations in a similar to Theorem~\ref{autrepr} way, and we will not do it here.

Since $\varphi_{2B}(\sigma_j)=R_j$ permutes $t_j$ and $t_{j+1}$ and fixes $t_i$ for $i\notin\{ j,j+1\}$, it is clear that $\varphi_{2B}(\sigma_j^2)=R_j^2$ fixes all $t_1,t_2,\dots,t_n$. Hence for $1 \leq i < j \leq n$ the automorphism $\varphi_{2B}(a_{i,j})$, where
\begin{equation}\label{puregen}
a_{i,j} = \sigma_{j-1} \sigma_{j-2} \dots \sigma_{i+1} \sigma_{i}^2 \sigma_{i+1}^{-1} \dots \sigma_{j-2}^{-1} \sigma_{j-1}^{-1}
\end{equation}
is the generator of $P_n$, fixes all $t_1,t_2,\dots,t_n$. Therefore
$$\varphi_{2Q}(a_{i,j})\left(\sum_{k=1}^n\alpha_ke_k\right)=\sum_{k=1}^n\alpha_k\varphi_{2Q}(a_{i,j})(e_k),$$
i.~e. the restriction of $\varphi_{2Q}$ to $P_n$ gives a representation $B_n\to {\rm Aut}(M)={\rm GL}_n(K)$ (here we write $M$ instead of $X$ in order to underline that $M$ is a module, while $X$ is an abelian group).

In order to prove that the restriction of the representation $\varphi_{2B}$ to $P_n$ coincides with the Gassner representation it is enough to prove that $\varphi_{2B}(a_{i,j})=\varphi_G(a_{i,j})$ for all $1 \leq i < j \leq n$, where $a_{i,j}$ is the generator of $P_n$ given by equality (\ref{puregen}). We will prove this fact using induction on $j-i$ and equalities (\ref{gassner}).

The basis of induction ($j=i+1$) is simple. From equality (\ref{trylin}) follows that
$$\varphi_{2B}(a_{i,i+1})=\varphi_{2B}(\sigma_i^2)=R_i^2 :\begin{cases}
 e_i \mapsto  (1 - t_{i}+t_it_{i+1}) e_i + t_i(1-t_i) e_{i+1}, \\
e_{i+1} \mapsto (1-t_{i+1})e_i+t_ie_{i+1}.
\end{cases}$$
Comparing the last equality with equality~(\ref{gassner}) we see that $\varphi_{2B}(a_{i,i+1})=\varphi_G(a_{i,i+1})$, and the basis of induction is proved.

In order to prove the induction step note that $R_j^{-1}$ acts by the following rule
\begin{equation}\label{trylininv}
R_j^{-1} :\begin{cases}
 e_j\mapsto e_{j+1}, \\
  e_{j+1} \mapsto t_{j+1}^{-1}e_j+t_{j+1}^{-1}(t_j-1)e_{j+1},  \\
t_j \mapsto  t_{j+1}, \\
t_{j+1} \mapsto t_j.
\end{cases}
\end{equation}
Suppose that we proved that $\varphi_{2B}(a_{i,j})=\varphi_G(a_{i,j})$ for $1\leq i<j\leq n$, and let us prove that $\varphi_{2B}(a_{i,j+1})=\varphi_G(a_{i,j+1})$. By the induction conjecture we have
\begin{equation}\label{step}\varphi_{2B}(a_{i,j+1})=\varphi_{2B}(\sigma_ja_{i,j}\sigma_{j}^{-1})=\varphi_{2B}(\sigma_j)\varphi_{2B}(a_{i,j})\varphi_{2B}(\sigma_j^{-1})=R_j\varphi_{G}(a_{i,j})R_j^{-1}.
\end{equation}
Using induction conjecture let us calculate the images $R_j\varphi_G(a_{i,j})R_j^{-1}(e_k)$ for all $k=1,2,\dots,n$. If $k<i$ or $k>j+1$, then it is clear that
\begin{equation}\label{gstep1}
R_j\varphi_G(a_{i,j})R_j^{-1}(e_k)=e_k
\end{equation}
since both $R_j$ and $\varphi_G(a_{i,j})$ fix $e_k$. For $k=i$ we have
\begin{align}
\notag R_j\varphi_G(a_{i,j})R_j^{-1}(e_i)&=\varphi_G(a_{i,j})R_j^{-1}(e_i)\\
\notag&=R_j^{-1}((1-t_i+t_it_j)e_i+t_i(1-t_i)e_j)\\
\label{gstep2}&=(1-t_i+t_it_{j+1})e_i+t_i(1-t_i)e_{j+1}.
\end{align}
For $i<k<j$ we have
\begin{align}
\notag R_j\varphi_G(a_{i,j})R_j^{-1}(e_k)&=\varphi_G(a_{i,j})R_j^{-1}(e_k)\\
\notag&=R_j^{-1}((1-t_{k})(1-t_j)e_i+e_k+(1-t_{k})(t_i-1)e_j)\\
\label{gstep4}&=(1-t_{k})(1-t_{j+1})e_i+e_k+(1-t_{k})(t_i-1)e_{j+1}.
\end{align}
For $k=j$ we have
\begin{align}
\notag R_j\varphi_G(a_{i,j})R_j^{-1}(e_j)&=\varphi_G(a_{i,j})R_j^{-1}((1-t_{j+1})e_j+t_je_{j+1})\\
\notag&=R_j^{-1}((1-t_{j+1})((1-t_j)e_i+t_ie_j)+t_je_{j+1})\\
\notag&=(1-t_{j})((1-t_{j+1})e_i+t_ie_{j+1})+t_{j+1}(t_{j+1}^{-1}e_{j}+t_{j+1}^{-1}(t_j-1))e_{j+1}\\
\label{gstep3}&=(1-t_{j})(1-t_{j+1})e_i+e_j+(1-t_j)(t_i-1)e_{j+1}.
\end{align}
Note that equalities (\ref{gstep4}) and (\ref{gstep3})  coincide. Finally, for $k=j+1$ we have
\begin{align}
\notag R_j\varphi_G(a_{i,j})R_j^{-1}(e_{j+1})&=\varphi_G(a_{i,j})R_j^{-1}(e_{j})\\
\notag &=R_j^{-1}((1-t_j)e_i+t_ie_j)\\
\label{gstep5} &=(1-t_{j+1})e_i+t_ie_{j+1}.
\end{align}
Comparing equalities (\ref{gstep1}), (\ref{gstep2}), (\ref{gstep4}), (\ref{gstep3}), (\ref{gstep5}) with equality (\ref{gassner}) we see that the equality $R_j\varphi_G(a_{i,j})R_j^{-1}(e_k)=\varphi_G(a_{i,j+1})$ holds for all $k=1,2,\dots,n$. From equality~(\ref{step}) follows that $\varphi_G(a_{i,j+1})=\varphi_{2B}(a_{i,j+1})$, so, the induction step is proved.
\end{proof}

Note that if we put $t_1=t_2=\dots=t_n=t$, then the representation $\varphi_{2B}$ is the Burau representation $\varphi_B$. So, the representation $\varphi_{2B}$ extends both the Burau and the Gassner representations.

Despite the fact that the map $R_j$ from (\ref{trylin}) has the same form as the map $R_j$ from~(\ref{fj}), Theorem~\ref{glin} does not follow from Theorem~\ref{autrepr}: the elements $e_1,e_2,\dots,e_n$, $t_1,t_2,\dots,t_n$ do not form the generating set of $X$ (remember that we think about the set $\{t_1,t_2,\dots,t_n\}$ as about a subset of $X$), and in order to extend the map $R_j$ to the automorphism of $X$ we need to use formula (\ref{ext}). At the same time, in Theorem~\ref{autrepr} the map $R_j$ was extended to the automorphism of $X$ just by action on the generators.

\subsection{Representations of $VB_n$ and $VP_n$}
The representation of the braid group
$$\varphi_{2B}:B_n\to{\rm Aut}\left(\left(\mathbb{Z}[t_1^{\pm1},t_2^{\pm1},\dots,t_n^{\pm1}]\right)^n\right)$$
from Theorem~\ref{glin} can be extended to the representation of the virtual braid group using the virtual $3$-switch $(S_{3B},V_{3B})$ instead of the $2$-switch $S_{2B}$. Let $n\geq 2$ be a positive integer, and $M$ be the free left module with the free basis $e_1, e_2, \ldots, e_n$ over the ring $K = \mathbb{Z}[t_1^{\pm 1},t_2^{\pm 1}, \dots, t_n^{\pm 1},q_1^{\pm 1},q_2^{\pm 1}, \dots, q_n^{\pm 1}]$. Denote by $X$ the additive group of $M$, by $X_0$ the subgroup of $X$ generated by $e_1,e_2,\dots,e_n$, by $X_1=\{t_1,t_2,\dots,t_n\}$, and by $X_2=\{q_1,q_2,\dots,q_n\}$. Let $(S_{3B}, V_{3B})$ be the virtual $3$-switch on $X$ from Proposition~\ref{Burgen3}
\begin{align}
\notag S_{3B}(a,b;x,y;p,q)&=((1 - y) a + x b, a; y, x;q,p),\\
\notag V_{3B}(a,b;x,y;p,q)&=(pb,q^{-1}a;y,x;q,p)
\end{align}
for $a,b\in X$, $x,y\in X_1$, $p,q\in X_2$. For $j=1,2,\dots,n-1$ denote by $R_j$, $G_j$ the following maps from $\{e_1,e_2,\dots,e_n,t_1,t_2,\dots,t_n,q_1,q_2,\dots,q_n\}$ to $X$
\begin{align}\label{vtrylin}
R_j :\begin{cases}
 e_j \mapsto  (1 - t_{j+1}) e_j + t_j e_{j+1}, \\
  e_{j+1} \mapsto e_j,  \\
t_j \mapsto  t_{j+1}, \\
t_{j+1} \mapsto t_j,\\
q_j\mapsto q_{j+1},\\
q_{j+1}\mapsto q_j,
\end{cases}&&
G_j :\begin{cases}
 e_j \mapsto  q_je_{j+1}, \\
  e_{j+1} \mapsto q_{j+1}^{-1}e_j,  \\
t_j \mapsto  t_{j+1}, \\
t_{j+1} \mapsto t_j,\\
q_j\mapsto q_{j+1},\\
q_{j+1}\mapsto q_j,
\end{cases}
\end{align}
The maps $R_j$, $G_j$ from (\ref{vtrylin}) are the same as the maps $R_j$, $G_j$ from equalities (\ref{fj}), (\ref{gj}) if we denote by $x^0_i=e_i$, $x^1_i=t_i$, $x^2_i=q_i$ for $i=1,2,\dots,n$. Similarly to equality (\ref{ext}) the maps $R_j$, $G_j$ induce automorphisms of $X$ by the rule
\begin{align}
\notag R_j\left(\sum_{k=1}^n\alpha_ke_k\right)=\sum_{k=1}^nR_j(\alpha_k)R_j(e_k),&&
\notag G_j\left(\sum_{k=1}^n\alpha_ke_k\right)=\sum_{k=1}^nG_j(\alpha_k)G_j(e_k),
\end{align}
where $\alpha_1,\alpha_2,\dots,\alpha_n\in K$, and $R_j(\alpha_k)=G_j(\alpha_k)$ is the image of $\alpha_k$ given by permutation $(t_j,~t_{j+1})(q_j,~q_{j+1})$ induced by $R_j, G_j$. Denote by $\varphi_{3B}$ the map from the set of generators $\{\sigma_1,\sigma_2,\dots,\sigma_{n-1},\rho_1,\rho_2,\dots,\rho_{n-1}\}$ of $VB_n$ to ${\rm Aut}(X)$ which maps $\sigma_j,\rho_j$ to $R_j$, $G_j$, respectively, for $j=1,2,\dots,n-1$.
\begin{theorem}\label{vvvglin}The map $\varphi_{3B} : \{\sigma_1,\sigma_2,\dots,\sigma_{n-1},\rho_1,\rho_2,\dots,\rho_{n-1}\} \to {\rm Aut}(X)$ induces a representation of the virtual braid group $VB_n$. The restriction of $\varphi_{3B}$ to the pure virtual braid group $VP_n$ is a linear representation $\varphi_{3B}:VP_n \to {\rm GL}_n(K)$.
\end{theorem}
\begin{proof}The proof is the same as the proof of the first part of Theorem~\ref{glin}.
\end{proof}
If in formulas~(\ref{vtrylin}) we put $q_1=q_2=\dots=q_n=0$, then the map $R_j$ from equality~(\ref{vtrylin}) becomes the same as the map $R_j$ from equality (\ref{trylin}), therefore the representation $\varphi_{3B}:VB_n\to{\rm Aut}\left(\left(\mathbb{Z}[t_1^{\pm1},t_2^{\pm1},\dots,t_n^{\pm1},q_1^{\pm1},q_2^{\pm1},\dots,q_n^{\pm1}]\right)^n\right)$ extends the representation $\varphi_{2B}:B_n\to{\rm Aut}\left(\left(\mathbb{Z}[t_1^{\pm1},t_2^{\pm1},\dots,t_n^{\pm1}]\right)^n\right)$. Hence, due to Theorem~\ref{glin}, the induced representation
$$\varphi_{3B}:VP_n\to{\rm GL}_n\left(\mathbb{Z}[t_1^{\pm1},t_2^{\pm1},\dots,t_n^{\pm1},q_1^{\pm1},q_2^{\pm1},\dots,q_n^{\pm1}]\right)$$
extends the Gassner representation $\varphi_G$. It is not known if $\varphi_G$ is faithful or not. In the following proposition we prove that the representation $\varphi_{3B}$ which extends the Gassner representation has non-trivial kernel.
\begin{prop} For $n\geq 3$ the representation
$$\varphi_{3B}:VP_n\to{\rm GL}_n\left(\mathbb{Z}[t_1^{\pm1},t_2^{\pm1},\dots,t_n^{\pm1},q_1^{\pm1},q_2^{\pm1},\dots,q_n^{\pm1}]\right)$$
has non-trivial kernel.
\end{prop}
\begin{proof}It is clear that it is enough to prove that the representation $$\varphi_{3B}:VP_3\to{\rm GL}_n\left(\mathbb{Z}[t_1^{\pm1},t_2^{\pm1},t_3^{\pm1},q_1^{\pm1},q_2^{\pm1},q_3^{\pm1}]\right)$$
has non-trivial kernel, since the elements from the kernel for $n=3$  belong to the kernel for arbitrary $n$. The group $VP_3$ has a subgroup $VP_3^{+}$ which has three generators
\begin{align}
\notag \lambda_{1,2} = \rho_1 \sigma_1^{-1},&&\lambda_{1,3} = \rho_2 \lambda_{1,2} \rho_2,&&\lambda_{2,3} = \rho_2 \sigma_2^{-1},
\end{align}
and one defining relation
$$\lambda_{1,2}\lambda_{1,3}\lambda_{2,3}=\lambda_{2,3}\lambda_{1,3}\lambda_{1,2}$$
(see, for example, \cite{Bar-0}). Using direct calculations it is easy to see that
\begin{align}
\notag R_1^{-1}:\begin{cases}
  e_{1} \mapsto e_2,  \\
  e_2 \mapsto t_2^{-1} e_{1} + t_2^{-1} (t_1 - 1) e_2,\\
t_1 \mapsto  t_{2}, \\
t_{2} \mapsto t_1,  \\
q_1 \mapsto  q_{2}, \\
q_{2} \mapsto q_1,  \\
\end{cases}&&
G_1:\begin{cases}
  e_{1} \mapsto q_1 e_2,\\
  e_2 \mapsto q_2^{-1} e_{1}, \\
t_1 \mapsto  t_{2}, \\
t_{2} \mapsto t_1,  \\
q_1 \mapsto  q_{2}, \\
q_{2} \mapsto q_1,  \\
\end{cases}\\
\notag R_2^{-1}:\begin{cases}
  e_{2} \mapsto e_3,  \\
  e_3 \mapsto t_3^{-1} e_{2} + t_3^{-1} (t_2 - 1) e_3, \\
t_2 \mapsto  t_{3}, \\
t_{3} \mapsto t_2,  \\
q_2 \mapsto  q_{3}, \\
q_{3} \mapsto q_2,  \\
\end{cases}&&
G_2 :\begin{cases}
  e_{2} \mapsto q_2 e_3,  \\
  e_3 \mapsto q_3^{-1} e_2, \\
t_2 \mapsto  t_{3}, \\
t_{3} \mapsto t_2,  \\
q_2 \mapsto  q_{3}, \\
q_{3} \mapsto q_2,  \\
\end{cases}
\end{align}
From these equalities and direct calculations follows that
\begin{align}
\label{l12}
\varphi_{3B}(\lambda_{1,2})=G_1R_1^{-1}&:
\begin{cases}
  e_1 \mapsto q_2( t_2^{-1} e_{1} + t_2^{-1} (t_1 - 1) e_2), \\
  e_{2} \mapsto q_1^{-1} e_2,  \\
  e_{3} \mapsto e_3,  \\
\end{cases}\\
\label{l13}
\varphi_{3B}(\lambda_{1,3})=G_2\varphi_{3B}(\lambda_{1,2})G_2&:
\begin{cases}
  e_1 \mapsto q_3( t_3^{-1} e_{1} + t_3^{-1} (t_1 - 1) q_2 e_3), \\
  e_{2} \mapsto e_2,  \\
  e_{3} \mapsto q_1^{-1} e_3,  \\
\end{cases}\\
\label{l23}
\varphi_{3B}(\lambda_{2,3})=G_2R_2^{-1}&:
\begin{cases}
  e_{1} \mapsto e_1,  \\
  e_2 \mapsto q_3( t_3^{-1} e_{2} + t_3^{-1} (t_2 - 1) e_3), \\
  e_{3} \mapsto q_2^{-1} e_3.  \\
\end{cases}
\end{align}
From equalities (\ref{l12}), (\ref{l13}), (\ref{l23}) we see that the matrices of the linear transformations $\varphi_{3B}(\lambda_{1,2})$, $\varphi_{3B}(\lambda_{1,3})$, $\varphi_{3B}(\lambda_{2,3})$ are upper triangular, therefore the group $\varphi_{3B}(VP_3^+)$ is solvable. However, $VP_3^{+}$ has three generators and one relation, therefore by Magnus theorem \cite[Section~4.4]{MKS} any two elements from the set $\{\lambda_{1,2},  \lambda_{1,3},\lambda_{2,3}\}$
generate non-abelian free group. Therefore the induced representation
$$\varphi_{3B}:VP_3^+\to {\rm GL}_n(\mathbb{Z}[t_1^{\pm1},t_2^{\pm1},t_3^{\pm1},q_1^{\pm1},q_2^{\pm1},q_3^{\pm1}])$$
has a non-trivial kernel
\end{proof}

{\small

\medskip

\medskip

\noindent
Valeriy Bardakov\\
Tomsk State University, pr. Lenina 36, 634050 Tomsk, Russia,\\
Sobolev Institute of Mathematics, Acad. Koptyug avenue 4, 630090 Novosibirsk, Russia,\\
Novosibirsk State University, Pirogova 1, 630090 Novosibirsk, Russia,\\
Novosibirsk State Agricultural University, Dobrolyubova 160, 630039 Novosibirsk, Russia,\\
bardakov@math.nsc.ru
~\\
~\\
Timur Nasybullov\\
KU Leuven KULAK, Etienne Sabbelaan 53, 8500 Kortrijk, Belgium,\\
timur.nasybullov@mail.ru
}


\begin{thebibliography}{HD}
\bibitem{Abd}
 M.~Abdulrahim, A faithfulness criterion for the Gassner representation of the pure braid group, Proc. Amer. Math. Soc., V.~125, N.~5, 1997, 1249--1257.
\bibitem{A} M.~Albar, D.~Johnson, The centre of the circular braid group, Math. Jpn., V.~30, 1985, 641--645.
\bibitem{Bar-0}
V.~Bardakov, The virtual and universal braids, Fund. Math., V.~184, 2004,  1--18.
 \bibitem{Bar-1} V.~Bardakov, Virtual and welded links and their invariants, Sib. Elektron. Mat. Izv., V.~2, 2005, 196--199.
 \bibitem{Bar-2}
 V.~Bardakov, Extending representations of braid groups to the automorphism groups of free groups, J. Knot Theory Ramifications, V. 14, N. 8, 2005, 1087--1098.
 \bibitem{BB}
V.~Bardakov, P.~Bellingeri, Combinatorial properties of virtual braids, Topology Appl., V.~156, N.~6, 2009, 1071--1082.
 \bibitem{BarMikNes}
V.~Bardakov, Yu.~Mikhalchishina, M.~Neshchadim, Representations of virtual braids by automorphisms and virtual knot groups, J.~Knot Theory Ramifications, V.~26, N.~1, 2017, 1750003.
\bibitem{BarMikNes2}
V.~Bardakov, Yu.~Mikhalchishina, M.~Neshchadim, Virtual link groups, Sib. Math. J., V.~58, N.~5, 2017, 765--777.
\bibitem{BarNas}
V.~Bardakov, T.~Nasybullov, Embeddings of quandles into groups, J.~Algebra App.,  https://doi.org/10.1142/S0219498820501364.
\bibitem{BN}
V.~Bardakov,  M.~Neshchadim, On a representation of virtual braids by automorphisms, Algebra Logic, V.~56, N.~5, 2017, 355--361.
\bibitem{Big}
S.~Bigelow, The Burau representation is not faithful for $n = 5$, Geom. Topol., V.~3, 1999, 397--404.
\bibitem{Big2}
S.~Bigelow, Braid groups are linear, J. Amer. Math. Soc., V.~14, N.~2, 2001, 471--486.
\bibitem{Bir}
J.~Birman, Braids, links, and mapping class groups, Annals of Math. Studies 82, Princeton University Press, 1974.
\bibitem{BDGGHN}
H.~Boden, E.~Dies, A.~Gaudreau, A.~Gerlings, E.~Harper, A.~Nicas, Alexander
invariants for virtual knots, J. Knot Theory Ramifications, V.~24, N.~3,  2015, 1550009.
\bibitem{Carter}
J.~Carter, A survey of quandle ideas, Introductory lectures on knot theory, Ser. Knots Everything, World Sci. Publ., Hackensack, NJ, V.~46,  2012, 22--53.
\bibitem{CSWES}
J.~Carter, D.~Silver, S.~Williams, M.~Elhamdadi, M.~Saito, Virtual knot invariants from group biquandles and their cocycles,
J. Knot Theory Ramifications, V.~18, N.~7, 2009, 957--972.
\bibitem{Dri}
 V.~Drinfel'd, On some unsolved problems in quantum group theory, Quantum groups (Leningrad, 1990), 1-8,
Lecture Notes in Math., 1510, Springer, Berlin, 1992.
\bibitem{ElhNel}
M.~Elhamdadi, S.~Nelson, Quandles--an introduction to the algebra of knots, Student Mathematical Library, V.~74, American Mathematical Society, Providence, RI, 2015.
\bibitem{FMK}
R.~Fenn, M.~Jordan-Santana, L.~Kauffman, Biquandles and virtual links, Topology Appl., V.~145, N.~1--3, 2004, 157--175.
\bibitem{F}
R. Fenn,  Tackling the trefoils, J. Knot Theory Ramifications, V.~21, N. 13, 2012, 1240004.
 \bibitem{Gas}
 B.~Gassner, On braid groups, Abh. Math. Sem. Univ. Hamburg, V.~25, 1961, 10--22.
\bibitem{GuaVen}
L. Guarnieri, L. Vendramin, Skew braces and the Yang-Baxter equation, Math. Comp., V.~86, N.~307, 2017, 2519--2534.
\bibitem{HK}
D. Hrencecin, L.~Kauffman,  Biquandles for virtual knots, J. Knot Theory Ramifications, V.~16, N.~10, 2007, 1361--1382.
\bibitem{Joy} D.~Joyce, A classifying invariant of knots, the knot quandle, J. Pure Appl. Algebra, V.~23, 1982, 37--65.
\bibitem{Kau}
L.~Kauffman, Virtual knot theory, Eur. J. Comb., V.~20, 1999, 663--690.
\bibitem{KM}
L.~Kauffman, V.~Manturov,  Virtual biquandles, Fund. Math., V.~188, 2005, 103--146.
\bibitem{KenPei}
R.~Kent, D.~Peifer, A geometric and algebraic description of annular braid groups, J. Algebra Comput., V.~12, N.~1-2, 2002, 85--97. 
\bibitem{Knu}
K.~Knudson, On the kernel of the Gassner representation, Arch. Math., V.~85, N.~2, 2005, 108--117.
\bibitem{kt}
The Kourovka notebook, Unsolved problems in group theory. Edited by V.~D.~Mazurov and E.~I.~Khukhro, 19-th. ed.. Russian Academy of Sciences Siberian Division. Institute of Mathematics, Novosibirsk, 2019.
\bibitem{Kra}
D.~Krammer, Braid groups are linear, Ann. of Math., V.~155, N.~1, 2002, 131--156.
\bibitem{Law}
R.~Lawrence,  Homological representations of the Hecke algebra, Comm. Math. Phys., V.~135, N.~1, 1990, 141--191.
\bibitem{MKS}
W.~Magnus, A.~Karrass, D.~Solitar, Combinatorial Group Theory: Presentations of Groups in Terms
of Generators and Relations, Chelmsford, MA: Courier Corporation, 2004.
\bibitem{Man2}
V.~Manturov, On invariants of virtual links, Acta Appl. Math., V.~72, N.~3, 2002, 295--309.
\bibitem{Mat} S. Matveev, Distributive groupoids in knot theory, (in Russian), Mat. Sb. (N.S.), V.~119(161), N.~1(9), 1982, 78--88.
\bibitem{Nas}
T.~Nasybullov, Connections between properties of the additive and the multiplicative groups of a two-sided skew brace, J.~Algebra, https://doi.org/10.1016/j.jalgebra.2019.05.005.
\bibitem{Nis}
V. Nisnewitsch,  \"{U}ber Gruppen die durch Matrizen \"{u}ber einem kommutativen Feld isomorph darstellbar
sind (Russian, German summary), Mat. Sb., V.~8, 1940, 395--403.
\bibitem{Nos}
T.~Nosaka, Quandles and topological pairs, 
Symmetry, knots, and cohomology, Springer Briefs in Mathematics, Springer, Singapore, 2017. 
\bibitem{R}
 L.~Rabenda, M\'emoire de DEA (Master thesis), Universit\'e de Bourgogne, 2003.
\bibitem{SilWil}
D.~Silver, S.~Williams, Alexander groups and virtual links, J.~Knot Theory Ramifications, V.~10, N.~1, 2001, 151--160.
\bibitem{Ver}
V.~Vershinin, On homology of virtual braids and Burau representation, J.~Knot Theory Ramifications, V. 10, N. 5, 2001, 795--812.
\end{thebibliography}
\end{document}